\documentclass[10pt]{article}
\usepackage{amssymb,amsmath,amsthm,youngtab}
\usepackage{verbatim}
\usepackage{graphicx}
\newtheorem*{definition}{Definition}
\newtheorem{theorem}{Theorem}
\newtheorem{lemma}[theorem]{Lemma}
\newtheorem{conjecture}[theorem]{Conjecture}
\newtheorem{proposition}[theorem]{Proposition}
\newtheorem{claim}{Claim}
\newtheorem{corollary}[theorem]{Corollary}

\DeclareMathOperator{\Cay}{Cay}
\DeclareMathOperator{\sgn}{sgn}

\begin{document}
\title{Forbidding just one intersection,\\for permutations.}
\author{David Ellis\footnote{School of Mathematical Sciences, Queen Mary, University of London, Mile End Road, London, E1 4NS, United Kingdom. E-mail: \texttt{D.Ellis@qmul.ac.uk.}}}
\date{October 2013}
\maketitle

\begin{abstract}
We prove that for \(n\) sufficiently large, if $\mathcal{A}$ is a family of permutations of \(\{1,2,\ldots,n\}\) with no two permutations in $\mathcal{A}$ agreeing exactly once, then $|\mathcal{A}| \leq (n-2)!$, with equality holding only if \(\mathcal{A}\) is a coset of the stabilizer of 2 points. We also obtain a Hilton-Milner type result, namely that if \(\mathcal{A}\) is such a family which is not contained within a coset of the stabilizer of 2 points, then it is no larger than the family
\begin{eqnarray*}
\mathcal{B} & = & \{\sigma \in S_{n}:\ \sigma(1)=1,\sigma(2)=2,\ \#\{\textrm{fixed points of }\sigma \geq 5\} \neq 1\}\\
&& \cup \{(1\ 3)(2\ 4),(1\ 4)(2\ 3),(1\ 3\ 2\ 4),(1\ 4\ 2\ 3)\}
\end{eqnarray*}
We conjecture that for \(t \in \mathbb{N}\), and for \(n\) sufficiently large depending on \(t\), if $\mathcal{A}$ is family of permutations of $\{1,2,\ldots,n\}$ with no two permutations in $\mathcal{A}$ agreeing exactly \(t-1\) times, then \(|\mathcal{A}| \leq (n-t)!\), with equality holding only if \(\mathcal{A}\) is a coset of the stabilizer of \(t\) points. This can be seen as a permutation analogue of a conjecture of Erd\H{o}s on families of $k$-element sets with a forbidden intersection, proved by Frankl and F\"uredi in \cite{franklfuredi}.
\end{abstract}

\section{Introduction}
Let \(X\) be an \(n\)-element set, and let \(X^{(k)}\) denote the collection of all \(k\)-element subsets of \(X\). We say a family \(\mathcal{A} \subset X^{(k)}\) is \(t\)-{\em intersecting} if any two sets in $\mathcal{A}$ share at least \(t\) elements, i.e. \(|x\cap y| \geq t\) for any \(x,y \in \mathcal{A}\). Erd\H{o}s, Ko and Rado \cite{tekr} proved in 1961 that if \(n\) is sufficiently large depending on \(k\) and \(t\), and \(\mathcal{A} \subset X^{(k)}\) is \(t\)-intersecting, then \(|\mathcal{A}| \leq {n-t \choose k-t}\), with equality holding only if \(\mathcal{A}\) is the family of all \(k\)-sets containing some fixed \(t\)-element subset of \(X\).

In \cite{erdosconj}, Erd\H{o}s asked what happens if we weaken the condition, and just forbid an intersection of size {\em exactly} \(t-1\). Frankl and F\"uredi \cite{franklfuredi} proved that for \(k \geq 2t\) and for \(n\) sufficiently large depending on \(k\), if $\mathcal{A} \subset X^{(k)}\) such that no two sets in $\mathcal{A}$ have intersection of size exactly \(t-1\), then $|\mathcal{A}| \leq {n-t \choose k-t}$, with equality holding only if \(\mathcal{A}\) is the family of all \(k\)-sets containing some fixed \(t\)-element subset of \(X\).

In this paper, we consider analogues of these problems for the symmetric group \(S_{n}\), the group of all permutations of \(\{1,2,\ldots,n\}=:[n]\). We say that a family of permutations \(\mathcal{A} \subset S_{n}\) is \(t\)-{\em intersecting} if any two permutations in $\mathcal{A}$ agree on at least $t$ points --- in other words, for all \(\sigma,\tau \in \mathcal{A}\), we have \(\#\{i:\ \sigma(i)=\tau(i)\} \geq t\).

Deza and Frankl \cite{dezafrankl} proved in 1977 that if $\mathcal{A} \subset S_n$ is 1-intersecting, then $|\mathcal{A}| \leq (n-1)!$. The case of equality turned out to be somewhat harder than one might expect; this was resolved in 2003 by Cameron and Ku \cite{cameron}, and independently by Larose and Malvenuto \cite{larose}, who proved that if $\mathcal{A} \subset S_n$ is an intersecting family of size $(n-1)!$, then $\mathcal{A}$ is a coset of the stabiliser of a point.

Deza and Frankl conjectured in \cite{dezafrankl} that for any $t \in \mathbb{N}$, if \(n\) is sufficiently large depending on \(t\), and \(\mathcal{A} \subset S_{n}\) is \(t\)-intersecting, then \(|\mathcal{A}| \leq (n-t)!\). This was proved in 2008, by the author and independently by Friedgut and Pilpel, using very similar techniques (specifically, eigenvalue methods, combined with the representation theory of $S_n$); we have written a joint paper, \cite{jointpaper}. We also proved that equality holds only if $\mathcal{A}$ is a {\em $t$-coset} of $S_n$ (meaning a coset of the stabiliser of $t$ points), again provided $n$ is sufficiently large depending on $t$.

Cameron and Ku \cite{cameron} conjectured that if \(\mathcal{A} \subset S_{n}\) is 1-intersecting, and $\mathcal{A}$ is not contained in any 1-coset, then $\mathcal{A}$ is no larger than the family
\[\{\sigma \in S_{n}:\ \sigma(1)=1,\ \sigma(j)=j \textrm{ for some }j>2\} \cup \{(1\ 2)\},\] 
which has size \((1-1/e+o(1))(n-1)!\). This was proved by the author in \cite{cameronkuconj}, using the representation theory of \(S_{n}\) combined with some combinatorial arguments. It can be seen as an analogue of the Hilton-Milner Theorem \cite{hiltonmilner} on 1-intersecting families of \(r\)-subsets of \(\{1,2,\ldots,n\}\). In \cite{dezafranklstability}, the author proved a generalization of the Cameron-Ku conjecture for \(t\)-intersecting families, namely that if \(\mathcal{A} \subset S_{n}\) is a \(t\)-intersecting family which is not contained within a coset of the stabilizer of \(t\) points, then \(\mathcal{A}\) is no larger than the family
\[\{\sigma: \sigma(i)=i\ \forall i \leq t,\ \sigma(j)=j\ \textrm{for some}\ j > t+1\} \cup \{(1\ t+1),\ldots,(t \ t+1)\}\]
which has size \((1-1/e+o(1))(n-t)!\). The proof uses similar ideas to in \cite{cameronkuconj}, but both the representation theory and the combinatorial arguments are somewhat more involved.

In this paper, we consider problem of forbidding just one intersection-size, for families of permutations. We say that two permutations $\sigma,\pi \in S_n$ {\em agree on exactly $k$ points} if $\#\{i \in [n]:\ \sigma(i)=\pi(i)\}=k$. We make the following conjecture.
\begin{conjecture}
\label{conj:all-t}
For any \(t \in \mathbb{N}\), and for \(n\) sufficiently large depending on \(t\), if $\mathcal{A} \subset S_n$ with no two permutations in $\mathcal{A}$ agreeing on exactly \(t-1\) points, then
\[|\mathcal{A}| \leq (n-t)!,\]
and equality holds only if \(\mathcal{A}\) is a \(t\)-coset of \(S_{n}\).
\end{conjecture}
This is a natural permutation analogue of the above-mentioned conjecture of Erd\H{o}s on families of $k$-element sets with a forbidden intersection.

Of course, a family of permutations in which no two permutations disagree everywhere is precisely a 1-intersecting family, so the $t=1$ case of Conjecture \ref{conj:all-t} holds for all $n$, by the above-mentioned results of Deza and Frankl and of Cameron and Ku.

In this paper, we prove the $t=2$ case of Conjecture \ref{conj:all-t}.
\begin{theorem}
\label{thm:main}
If $n$ is sufficiently large, and $\mathcal{A} \subset S_n$ is a family of permutations with no two permutations in $\mathcal{A}$ agreeing at exactly one point, then
$$|\mathcal{A}| \leq (n-2)!,$$
and equality holds only if $\mathcal{A}$ is a $2$-coset of $S_n$.
\end{theorem}

We also prove a corresponding stability result and a Hilton-Milner type result. We use similar techniques to in \cite{cameronkuconj} and \cite{jointpaper} --- namely, eigenvalue techniques, combined with the representation theory of $S_n$ --- but these techniques do not work quite so cleanly in the present case. Indeed, they are only capable of proving asymptotic results, and must be combined with stability arguments to prove the exact bound in Theorem \ref{thm:main}.

Unfortunately, for each $t \geq 3$ in Conjecture \ref{conj:all-t}, our techniques yield only a bound of $O((n-2)!)$. This remains true even if one uses `weighted' versions of Hoffman's bound, such as Lov\'asz' theta-function bound (see for example \cite{lovasz}), together with conjugation-invariant weightings. (While Lov\'asz' theta-function concerns arbitrary weightings, not just conjugation-invariant ones, conjugation-invariance is necessary if one wishes to use representation theory to analyse eigenvalues and eigenspaces, via Theorem \ref{thm:normalcayley} below.) It seems that new techniques will be required to solve the problem for \(t \geq 3\).

\section{Notation, background and tools}
In this section, we outline our notation, recall the tools we will use to prove our main results, and give some background on the representation theory of $S_n$.

From now on, we will often abbreviate the condition `no two permutations in \(\mathcal{A}\) agree at exactly one point' to `\(\mathcal{A}\) has no singleton intersection'.

Let $\Gamma_n$ be the graph on $S_n$ where two permutations \(\sigma\) and \(\tau\) are joined if and only if they agree at exactly one point. This is the Cayley graph on \(S_{n}\) generated by the set
\[\mathcal{E}_{n}=\{\sigma \in S_{n}:\ \sigma \textrm{ has exactly one fixed point}\}.\]
For \(n \in \mathbb{N}\), let \(d_{n}\) denote the number of derangements of \(n\), i.e. permutations in \(S_{n}\) without fixed points; the inclusion-exclusion formula yields the familiar identity
\[d_{n}=\sum_{i=0}^{n}(-1)^{i}\frac{n!}{i!} = (1/e+o(1))n!.\]
Note that
$$|\mathcal{E}_{n}|=nd_{n-1} = (1/e+o(1))n!,$$
so \(\Gamma_n\) is \(nd_{n-1}\)-regular.

Observe that a family \(\mathcal{A} \subset S_{n}\) in which no two permutations agree at exactly one point, is precisely an {\em independent set} in \(V(\Gamma_n)\), meaning a set of vertices of \(\Gamma_n\) with no edges of \(\Gamma_n\) between them. 

Our first step will be to apply Hoffman's eigenvalue bound to \(\Gamma_n\).
\begin{theorem}[Hoffman, \cite{hoffman}]
\label{thm:hoffman}
Let \(\Gamma\) be a \(d\)-regular graph on \(N\) vertices, whose adjacency matrix \(A\) has eigenvalues \(d=\lambda_{1} \geq \lambda_{2} \geq \ldots \geq \lambda_{N}\). Then if \(X \subset V(\Gamma)\) is an independent set, we have
\[|X| \leq \frac{-\lambda_{N}}{d-\lambda_{N}}N.\]
\end{theorem}

This will yield an approximate version of Theorem \ref{thm:main}. To apply Theorem \ref{thm:hoffman}, we will of course need to calculate the minimum eigenvalue of $\Gamma_n$. Since \(\mathcal{E}_{n}\) is a union of conjugacy-classes of \(S_{n}\), we may analyse the eigenvalues of $\Gamma$ using representation-theoretic techniques. Before outlining these, we first give some background on the representation theory of finite groups.

\subsection*{Background on the representation theory of finite groups}

If $G$ is a finite group, let $\mathbb{C}[G]$ denote the Euclidean space of all complex-valued functions on $G$, equipped with the inner product
$$\langle f,g \rangle = \frac{1}{|G|} \sum_{\sigma \in G} f(\sigma) \overline{g(\sigma)}\quad (f,g \in \mathbb{C}[G]),$$
and the corresponding Euclidean norm
$$||f||_2 = \sqrt{\frac{1}{|G|} \sum_{\sigma \in G} |f(\sigma)|^2} \quad (f \in \mathbb{C}[G]).$$
Recall that a {\em representation} of $G$ over $\mathbb{C}$ is a pair $(\rho,V)$, where $V$ is a finite-dimensional complex vector space, and $\rho$ is a homomorphism from $G$ to $GL(V)$, the group of all invertible linear endomorphisms of $V$. The {\em dimension} of the representation is the dimension of $V$. A representation $(\rho,V)$ is said to be {\em irreducible} if it has no proper subrepresentation, i.e. there is no proper subspace $V' \leq V$ such that $V'$ is $\rho(\sigma)$-invariant for all $\sigma \in G$. We say that two representations \((\rho,V)\) and \((\rho',V')\) are {\em isomorphic} if there is an invertible linear map \(\phi\colon V \to V'\) such that such that \(\phi(\rho(g)(v)) = \rho'(g)(\phi(v))\) for all \(g \in G\) and all \(v \in V\). In this case, we write $(\rho,V) \cong (\rho',V')$. 

Recall that, as a vector space, $\mathbb{C}[G]$ may be equipped with the {\em left-regular representation}, defined by
$$(\rho_{\textrm{reg}}(\pi) (f))(\sigma) = f(\pi^{-1} \sigma)\quad (f \in \mathbb{C}[G],\ \sigma,\pi \in G).$$

For any finite group $G$, we can choose a complete set \(\mathcal{R}\) of non-isomorphic complex irreducible representations of \(G\) --- i.e., a set of complex irreducible representations of $G$ containing exactly one member of each isomorphism class of complex irreducible representations of $G$. For each $\rho \in \mathcal{R}$, let $U_{\rho}$ denote the subspace of $\mathbb{C}[G]$ spanned by all isomorphic copies of $\rho$ in the left-regular representation. Then we have an orthogonal direct-sum decomposition
$$\mathbb{C}[G] = \bigoplus_{\rho \in \mathcal{R}} U_{\rho},$$
and $\dim(U_{\rho}) = (\dim(\rho))^2$ for all $\rho$.

We will use the following classical result to analyse the eigenvalues of $\Gamma$.

\begin{theorem}[Frobenius / Schur / Diaconis-Shahshahani \cite{diaconis}]
\label{thm:normalcayley}
Let \(G\) be a finite group, let \(X \subset G\) be an inverse-closed, conjugation-invariant subset of \(G\), let \(\Gamma = \Cay(G,X)\) be the Cayley graph on \(G\) with generating set \(X\), and let \(A\) be the adjacency matrix of \(\Gamma\). Let \(\mathcal{R}\) be a complete set of non-isomorphic complex irreducible representations of \(G\). Then each $U_{\rho}$ is an eigenspace of $\Gamma$, with corresponding eigenvalue 
\begin{equation}\label{eq:normalcayley}\lambda_{\rho} = \frac{1}{\dim(\rho)} \sum_{\sigma \in X} \chi_{\rho}(\sigma).\end{equation}
\end{theorem}
Here, $\chi_{\rho}$ denotes the {\em character} of $\rho$, i.e.
$$\chi_{\rho}(\sigma) = \textrm{Trace}(\rho(\sigma)).$$

\subsection*{Background and tools from the representation theory of $S_n$}

We will now give some brief background on the representation theory of $S_n$. For more detail, the reader may consult for example \cite{sagan}, or the exposition in \cite{jointpaper}.

As is well-known, there is an explicit one-to-one correspondence between irreducible representations of $S_n$ (up to isomorphism) and {\em partitions} of $n$.

\begin{definition} A \emph{partition} of \(n\) is a non-increasing sequence of positive integers summing to \(n\), i.e. a sequence $\alpha = (\alpha_1, \ldots, \alpha_l)$ with \(\alpha_{1} \geq \alpha_{2} \geq \ldots \geq \alpha_{l} \geq 1\) and \(\sum_{i=1}^{l} \alpha_{i}=n\).
\end{definition}
If $\alpha$ is a partition of $n$, we write \(\alpha \vdash n\). For example, \((3,2,2) \vdash 7\). We sometimes use the shorthand \((3,2,2) = (3,2^{2})\).
\begin{definition} The \emph{Young diagram} of $\alpha = (\alpha_1,\ldots,\alpha_l)$ is an array of $n$ cells, having $l$ left-justified rows where row $i$ contains $\alpha_i$ cells.
\end{definition}
For example, the Young diagram of the partition \((3,2^{2})\) is\\
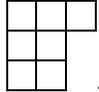
$$\yng(3,2,2).$$
For each partition \(\alpha\) of \(n\), we may define an irreducible representation \(\rho_{\alpha}\) of \(S_{n}\) called the {\em Schur module} of \(\alpha\); the Schur modules \(\{\rho_{\alpha}:\ \alpha \vdash n\}\) form a complete set of non-isomorphic irreducible complex representations of \(S_{n}\). We write the corresponding character as \(\chi_{\alpha}\), and the corresponding dimension \(\dim(\rho_{\alpha}) = f^{\alpha}\).

An {\em \(\alpha\)-tableau} is a Young diagram of shape \(\alpha\), each of whose cells contains a different number between 1 and \(n\). For example,
$$\young(713,52,46)$$
is a $(3,2^2)$-tableau. We say that two \(\alpha\)-tableaux are {\em row-equivalent} if they contain the same numbers in each row. A row-equivalence-class of \(\alpha\)-tableaux is called a {\em \(\alpha\)-tabloid}. Consider the natural action of \(S_n\) on the set of \(\alpha\)-tabloids, and let \(M^{\alpha}\) denote the induced permutation representation. We write \(\xi_{\alpha}\) for the character of \(M_{\alpha}\); the \(\xi_{\alpha}\) are called the {\em permutation characters} of \(S_n\).

We can express the irreducible characters in terms of the permutation characters using the {\em determinantal formula}: for any partition \(\alpha\) of \(n\),
\begin{equation}\label{eq:determinantalformula} \chi_{\alpha} = \sum_{\pi \in S_{n}} \sgn(\pi) \xi_{\alpha - \textrm{id}+\pi}.\end{equation}
Here, if \(\alpha = (\alpha_{1},\alpha_{2},\ldots,\alpha_{l})\), \(\alpha - \textrm{id}+\pi\) is defined to be the sequence
\[(\alpha_{1}-1+\pi(1),\alpha_{2}-2+\pi(2),\ldots,\alpha_{l}-l+\pi(l)).\]
If this sequence has all its entries non-negative, we let \(\overline{\alpha-\textrm{id}+\pi}\) be the partition of \(n\) obtained by reordering its entries, and we define \(\xi_{\alpha - \textrm{id}+\pi} = \xi_{\overline{\alpha-\textrm{id}+\pi}}\). If the sequence has a negative entry, we define \(\xi_{\alpha - \textrm{id}+\pi} = 0\).

For any partition $\alpha \vdash n$, and any permutation $\sigma \in S_n$,  $\xi_{\alpha}(\sigma)$ is the trace of the permutation representation $M^{\alpha}$ at $\sigma$, which is simply the number of $\alpha$-tabloids fixed by $\sigma$. For example, $\xi_{(n-1,1)}(\sigma)$ is the number of $(n-1,1)$-tabloids fixed by $\sigma$, which is precisely the number of fixed points of $\sigma$. It will be be convenient for us to express certain irreducible characters in terms of permutation characters, via the determinantal formula.

We need one final representation-theoretic tool.

\begin{definition}
Let $\alpha$ be a partition of $n$. The {\em transpose} of $\alpha$ is the partition $\alpha^t$ whose Young diagram is obtained by transposing that of $\alpha$. In other words, if the Young diagram of $\alpha$ has $k$ columns of heights $c_1 \geq c_2 \geq \ldots \geq c_k$, then $\alpha^{t} = (c_1,\ldots,c_k)$.
\end{definition}
\begin{lemma}
\label{lemma:sign}
For any partition $\alpha \vdash n$, we have
$$\rho_{\alpha^{t}} \cong \rho_{\alpha} \otimes \sgn,$$
where $\sgn$ denotes the (1-dimensional) sign-representation of $S_n$, and $\otimes$ denotes tensor product of representations. Hence,
$$\chi_{\alpha^{t}}(\sigma) = \chi_{\alpha}(\sigma) \sgn(\sigma) \quad \forall \sigma \in S_n.$$
\end{lemma}
This will also be a useful tool for calculations.

\subsection*{Preliminary results}
In this section, we list some preliminary results which will be useful in our proofs. 

Theorem \ref{thm:normalcayley} implies that if \(\Gamma = \Cay(S_{n},X)\) is a normal Cayley graph on $S_n$, then the eigenvalues of its adjacency matrix are given by
\begin{equation}\label{eq:evals}\lambda_{\alpha} = \frac{1}{f^{\alpha}} \sum_{\sigma \in X} \chi_{\alpha}(\sigma) \quad (\alpha \vdash n).\end{equation}
Note that an eigenvalue \(\lambda\) has geometric multiplicity
\[\sum_{\alpha \vdash n:\ \lambda_{\alpha}=\lambda}(f^{\alpha})^{2}.\]

In our case of $\Gamma = \Gamma_n$, we can bound all but 8 of these eigenvalues by $O((n-3)!)$, using only the fact that their geometric multiplicities are large. We will do this using the following simple identity.
\begin{lemma}
\label{lemma:trace}
If \(G\) is an \(N\)-vertex graph whose adjacency matrix \(A\) has eigenvalues \(\lambda_{1} \geq \ldots \geq \lambda_{N}\) (repeated with their multiplicities), then
\[\sum_{i=1}^{N} \lambda_{i}^{2} = 2e(G).\]
\end{lemma}
\begin{proof}
We have 
$$ \sum_{i=1}^{N} \lambda_{i}^{2} = \textrm{Trace}(A^2) = \textrm{Trace}(A^{\top}A) = \sum_{i,j}A_{i,j}^2 = \sum_{i,j} A_{i,j} = 2e(G).$$
\end{proof}
Combined with the fact that $\lambda_{\alpha}$ has multiplicity $\geq (f^{\alpha})^2$, this immediately implies the following.
\begin{lemma}
\label{lemma:cayley}
If $\Gamma = \Cay(S_n,X)$ is a normal Cayley graph on $S_n$, then its eigenvalues satisfy
$$|\lambda_{\alpha}| \leq \frac{\sqrt{|X|n!}}{f^{\alpha}}\quad (\alpha \vdash n).$$
\end{lemma}
We now recall a bound on the dimensions $f^{\alpha}$ from \cite{jointpaper}.
\begin{definition}
We say that a partition of \(n\) is \(k\)-{\em fat} if its Young diagram has first row of length at least \(n-k\), \(k\)-{\em tall} if its Young diagram has first column of height least \(n-k\), and \(k\)-{\em medium} otherwise.
\end{definition}
\begin{lemma}(See \cite{jointpaper}.)
For any \(k \in \mathbb{N}\), there exists a positive constant \(c_{k}\) such that for any \(n \in \mathbb{N}\), and any \(k\)-medium partition \(\alpha \vdash n\), we have
\[f^{\alpha} \geq c_{k} n^{k+1}.\]
\end{lemma}
In particular, any 2-medium $\alpha$ has $f^{\alpha} \geq c_2 n^3$. Combined with Lemma \ref{lemma:cayley}, and the fact that $|\mathcal{E}_n| = (1/e+o(1))n!$, this yields a bound on all the eigenvalues corresponding to 2-medium partitions.
\begin{corollary}
\label{corr:2medium}
If $\alpha$ is a 2-medium partition of $n$, then
$$|\lambda_{\alpha}| \leq \frac{\sqrt{1/e+o(1)}n!}{c_2n^3} \leq C (n-3)!,$$
where $C$ is an absolute constant.
\end{corollary} 

\section{An asymptotic result}
We are now ready to prove our asymptotic version of Theorem \ref{thm:main}.
\begin{proposition}
\label{prop:asymptoticbound}
Let \(\mathcal{A} \subset S_{n}\) be such that no two permutations in \(\mathcal{A}\) agree on exactly one point. Then
\[|\mathcal{A}| \leq (1+O(1/n))(n-2)!.\]
\end{proposition}
\begin{proof}
To prove this, we will simply calculate the minimum eigenvalue of $\Gamma_n$ and apply Theorem \ref{thm:hoffman} (Hoffman's bound). By Corollary \ref{corr:2medium}, all the 2-medium partitions $\alpha$ of $n$ have
\begin{equation} \label{eq:uniformbound} |\lambda_{\alpha}| \leq C(n-3)!.\end{equation}

There are 8 other partitions of $n$. Using equation (\ref{eq:evals}), combined with the determinantal formula (\ref{eq:determinantalformula}), one finds that these have eigenvalues as follows.\\

\begin{tabular}{l|l}
\(\alpha\) & \(\lambda_{\alpha}\)\\
\hline
\((n)\) & \(nd_{n-1}\)\\
\((1^{n})\) & \((-1)^{n-2}n(n-2)\)\\
\((n-1,1)\) & \(0\)\\
\((2,1^{n-2})\) & \(0\)\\
\((n-2,2)\) & \(-\frac{nd_{n-1}}{(n-1)(n-2)-2} (1+(-1)^{n}(n-2)/d_{n-1})\)\\
\((2,2,1^{n-4})\) & \((-1)^{n-1}(n-2)^{2}\)\\
\((n-2,1,1)\) & \(-\frac{nd_{n-1}}{(n-1)(n-2)}(1-(-1)^{n}(n-2)/d_{n-1})\)\\
\((3,1^{n-3})\) & \((-1)^{n}n(n-4)\)\\
\end{tabular}\\
\\

Note that, since $\chi_{(n)} \equiv 1$ is the trivial character, by (\ref{eq:evals}) we have
$$\lambda_{(n)} = |\mathcal{E}_n| = nd_{n-1}.$$

In the case $\alpha = (n-2,2)$, the determinantal formula yields
$$\chi_{(n-2,2)} = \xi_{(n-2,2)} - \xi_{(n-1,1)}.$$
Note that $\xi_{(n-2,2)}(\sigma) = \#\{x \in [n]^{(2)}:\ \sigma(x)=x\}$ is simply the number of pairs which are fixed by $\sigma$, and $\xi_{(n-1,1)}(\sigma)$ is the number of fixed points of $\sigma$. We have
$$f^{(n-2,2)} = \chi_{(n-2,2)}(\textrm{Id}) = \xi_{(n-2,2)}(\textrm{Id}) - \xi_{(n-1,1)}(\textrm{Id}) = {n \choose 2} - n.$$
Therefore, by (\ref{eq:evals}), we have
\begin{align*}
\lambda_{(n-2,2)} & = \frac{1}{{n \choose 2}-n} \sum_{\sigma \in \mathcal{E}_n} (\xi_{(n-2,2)}(\sigma)-\xi_{(n-1,1)}(\sigma))\\
& = \frac{1}{{n \choose 2}-n} \left(\sum_{ij \in [n]^{(2)}} \sum_{\sigma \in \mathcal{E}_n} 1\{\sigma\{i,j\} = \{i,j\}\} - nd_{n-1}\right)\\
& = \frac{1}{{n \choose 2}-n} \left({n \choose 2} (n-2)d_{n-3} - nd_{n-1}\right)\\
& = \frac{1}{{n \choose 2}-n} \left({n \choose 2} (n-2) \frac{d_{n-1}+ (-1)^{n-1}(n-2)}{(n-1)(n-2)} -nd_{n-1}\right)\\
& = -\frac{n}{2\left({n \choose 2}-n\right)} (d_{n-1}+(-1)^{n}(n-2))\\
& = -\frac{nd_{n-1}}{2\left({n\choose 2}-n\right)}(1+(-1)^{n}(n-2)/d_{n-1})\\
& = -\frac{nd_{n-1}}{(n-1)(n-2)-2} (1+(-1)^{n}(n-2)/d_{n-1}).
\end{align*}

The calculations of $\lambda_{(n-1,1)}$ and $\lambda_{(n-2,1^2)}$ are very similar.

For $\alpha = (1^n)$, we have $\chi_{(1^n)} = \sgn$, so by (\ref{eq:evals}), we have
$$\lambda_{(1^n)} = \sum_{\sigma \in \mathcal{E}_n} \sgn(\sigma) = n(e_{n-1} - o_{n-1}),$$
where $e_n$ and $o_n$ denote the number of respectively even and odd derangements in $S_n$. It is well-known (and can easily be proved by induction) that $e_n - o_n = (-1)^{n-1}(n-1)$ for all $n$. This yields
\[\lambda_{(1^{n})} = n(e_{n-1}-o_{n-1}) = (-1)^{n-2}n(n-2).\]

For the calculations of $\lambda_{(2,1^{n-2})}$, $\lambda_{(2,2,1^{n-4})}$ and $\lambda_{(3,1^{n-3})}$, we use Lemma \ref{lemma:sign}, combined with (\ref{eq:evals}) and the determinantal formula. For example,

 \begin{align*}
\lambda_{(2,2,1^{n-4})} & = \frac{1}{{n \choose 2}-n} \sum_{\sigma \in \mathcal{E}_n} \sgn(\sigma)(\xi_{(n-2,2)}(\sigma)-\xi_{(n-1,1)}(\sigma))\\
& = \frac{1}{{n \choose 2}-n} \left(\sum_{ij \in [n]^{(2)}} \sum_{\sigma \in \mathcal{E}_n} \sgn(\sigma)1\{\sigma\{i,j\} = \{i,j\}\} - n(e_{n-1}-o_{n-1})\right)\\
& = \frac{1}{{n \choose 2}-n} \left({n \choose 2} (n-2)(o_{n-3}-e_{n-3}) - n(n-2)(-1)^{n-2}\right)\\
& = \frac{1}{{n \choose 2}-n} \left({n \choose 2} (n-2) (-1)^{n-3}(n-4)+ n(n-2)(-1)^{n-1}\right)\\
& = -\frac{(-1)^{n-1}}{2\left({n \choose 2}-n\right)} (n(n-1)(n-2)(n-4)+2n(n-2))\\
& = -\frac{(-1)^{n-1}n(n-2)}{(n-1)(n-2)-2} ((n-1)(n-4)+2)\\
& = -\frac{(-1)^{n-1}n(n-2)}{n(n-3)} (n-2)(n-3)\\
& =  (-1)^{n-1}(n-2)^{2}.
\end{align*}

We see from equation (\ref{eq:uniformbound}) and the table following it that $\Gamma_n$ has
\[\lambda_N = -(1+O(1/n))\frac{nd_{n-1}}{n^{2}}.\]
Applying Theorem \ref{thm:hoffman} (Hoffman's bound) with \(d = nd_{n-1}\), \(N=n!\), and the above bound on $\lambda_N$, we see that any independent set $\mathcal{A} \subset V(\Gamma_n)$ has
\[|\mathcal{A}| \leq (1+O(1/n))(n-2)!,\]
proving Proposition \ref{prop:asymptoticbound}.
\end{proof}

From the above proof, we can also read off a two-family version of Proposition \ref{prop:asymptoticbound}. Recall the following `cross-independent' version of Hoffman's theorem.
\begin{theorem}
\label{thm:crosshoffman}
Let \(\Gamma\) be a \(d\)-regular graph on \(N\) vertices, whose adjacency matrix \(A\) has eigenvalues \(\lambda_{1}=d \geq \lambda_{2} \geq \ldots \geq \lambda_{N}\). Let \(\nu = \max(|\lambda_{2}|,|\lambda_{N}|)\). Then if \(X,Y \subset V(\Gamma)\) with \(xy \notin E(\Gamma)\) for any $x \in X$ and any $y \in Y$, we have:
\[\sqrt{|X||Y|} \leq \frac{\nu}{d+\nu}N.\]
\end{theorem}
(For a proof, see for example \cite{jointpaper}.) By equation (\ref{eq:uniformbound}) and the table following it, our graph $\Gamma_n$ has $\nu = |\lambda_N|$ for $n$ sufficiently large. Hence, Theorem \ref{thm:crosshoffman} immediately gives the following.

\begin{proposition}
\label{prop:cross}
If \(\mathcal{A},\mathcal{B} \subset S_n\) are families of permutations such that no permutation in \(\mathcal{A}\) agrees with any permutation in \(\mathcal{B}\) at exactly one point, then
\[|\mathcal{A}||\mathcal{B}| \leq (1+O(1/n))((n-2)!)^{2}.\]
\end{proposition}

This proposition will be used in the proof of our stability result in the next section.

\section{A stability result}
Our main aim in this section is to prove the following stability result, which says that a `large' family of permutation with no singleton intersection, is `almost' contained within a 2-coset.
\begin{theorem}
\label{thm:roughstability}
Let \(c >0\). Let \(\mathcal{A} \subset S_{n}\) with no singleton intersection, and with \(|\mathcal{A}| \geq c(n-2)!\). Then there exist \(i,j,k,l \in [n]\) such that
\[|\mathcal{A} \setminus \mathcal{A}_{i \mapsto j,k \mapsto l}| \leq K_{c}(n-3)!,\]
where $K_c >0$ depends only upon $c$.
\end{theorem}

In our proof, we will use the following additional notation. If $f = f(n,c)$ and $g = g(n,c)$ are functions of $n$ and $c$, we will write $f = O_c (g)$ to mean that for each $c >0$, there exists $K_c >0$ depending upon $c$ alone, such that $f(n,c) \leq K_c g(n,c)$ for all $n \in \mathbb{N}$. 

For two permutations \(\sigma,\tau \in S_{n}\), we write \(\sigma \cap \tau\) for the set \(\{i:\ \sigma(i)=\tau(i)\}\), i.e. for the set of points at which they agree. For \(\mathcal{A} \subset S_{n}\), and for distinct \(i_{1},\ldots,i_{l} \in [n]\) and distinct \(j_{1},\ldots,j_{l} \in [n]\), we write
\[\mathcal{A}_{i_{1} \mapsto j_{1}, i_{2} \mapsto j_{2}, \ldots, i_{l} \mapsto j_{l}} := \{\sigma \in \mathcal{A}: \sigma(i_{k})=j_{k}\ \forall k \in [l]\}.\]
Similarly, for \(\mathcal{A} \subset S_{n}\), and for \(S,T \subset [n]\), we write
\[\mathcal{A}_{S \to T} = \{\sigma \in \mathcal{A}:\ \sigma(S) \subset T\}\]
for the subset of permutations in \(\mathcal{A}\) which map \(S\) into \(T\).

If $H$ is an inner product space, and $W$ is a subspace of $H$, we write $P_W:H \to H$ for orthogonal projection onto $W$.

If $V$ is a finite set, and $X \subset V$, we write $\chi_{X} \in \mathbb{R}^V$ for the {\em characteristic vector} of $X$, meaning the vector whose $i$th coordinate is $1$ if $i \in X$ and 0 if $i \notin X$. We use $\mathbf{f}$ to denote the all-1's vector in $\mathbb{R}^V$. We regard $\mathbb{R}^V$ as an inner product space, equipped with the inner product
$$\langle x,y \rangle = \frac{1}{|V|} \sum_{i \in V} x_i y_i;$$
we let
$$||x||_2 = \sqrt{\frac{1}{|V|} \sum_{i \in V} x_i^2 }$$
denote the corresponding Euclidean norm.

We will need the following `stability version' of Hoffman's theorem.

\begin{lemma}
\label{lemma:stability-hoffman}
Let \(\Gamma = (V,E)\) be an $N$-vertex, \(d\)-regular graph whose adjacency matrix $A$ has eigenvalues \(d=\lambda_{1} \geq \lambda_{2} \geq \ldots \geq \lambda_{N}\). Let $M \in \{1,2,\ldots,N-1\}$, and let
\[U = \textrm{Span}\{\mathbf{f}\} \oplus \bigoplus_{i > M}\\ker(A-\lambda_i I).\]
Let $X \subset V$ be an independent set, and let \(\alpha = |X|/N\). Define
$$D = ||\chi_X - P_{U} (\chi_X)||_2,$$
i.e. $D$ is the Euclidean distance from $\chi_X$ to $U$. Then
\[D^2 \leq \frac{(1-\alpha)|\lambda_{N}| - d \alpha}{|\lambda_{N}|-|\lambda_{M}|}\alpha.\]
\end{lemma}

For completeness, we include a proof.
\begin{proof}
Let \(u_{1} = \mathbf{f}, u_{2},\ldots,u_{N}\) be an orthonormal basis of real eigenvectors of \(A\) corresponding to the eigenvalues \(\lambda_{1},\lambda_2,\ldots,\lambda_{N}\). Write
\[\chi_{X}=\sum_{i=1}^{N} \xi_{i} u_{i}\]
as a linear combination of the eigenvectors of \(A\). Note that
\begin{equation}\label{eq:measure} \xi_{1}=\langle \chi_X, \mathbf{f} \rangle = |X|/N = \alpha\end{equation}
and
\begin{equation} \label{eq:parseval} \sum_{i=1}^{N} \xi_{i}^{2} = \langle \chi_X, \chi_X \rangle = |X|/N = \alpha.\end{equation}
Then we have the crucial property:
\[0=\frac{1}{N} \sum_{x,y \in X}A_{x,y}= \langle \chi_{X},A \chi_{X} \rangle = \sum_{i=1}^{N} \lambda_{i} \xi_{i}^{2} \geq \lambda_{1}\xi_{1}^{2} + \lambda_{N} \sum_{i=M+1}^{N} \xi_{i}^{2} + \lambda_{M} \sum_{i=2}^{M}\xi_{i}^{2}.\]
Note that
\[\sum_{i=2}^{M} \xi_{i}^{2} = D^{2}\]
and
\[\sum_{i=M+1}^{N} \xi_{i}^{2} = \alpha - \alpha^{2} - D^{2},\]
using (\ref{eq:measure}) and (\ref{eq:parseval}). Hence,
\[0 \geq \lambda_{1} \alpha^{2} + \lambda_{N} (\alpha -\alpha^{2} -D^{2}) + \lambda_{M} D^{2}.\]
Rearranging, we obtain:
\[D^{2} \leq  \frac{(1-\alpha)|\lambda_{N}| - \lambda_{1} \alpha}{|\lambda_{N}|-|\lambda_{M}|}\alpha,\]
as required.
\end{proof}

We will also need the following two technical lemmas, which we prove using Proposition \ref{prop:cross}.

\begin{lemma}
\label{lemma:twoposvalues}
Let \(\mathcal{A},\mathcal{B} \subset S_{n}\) such that for any \(\sigma \in \mathcal{A}\) and \(\tau \in \mathcal{B}\),
\begin{itemize}
\item If \(\sigma(1)=\tau(1)\), then \(|\sigma \cap \tau| \neq 2\);
\item If \(\sigma(1)\neq \tau(1)\), then \(|\sigma \cap \tau| \neq 1\).
\end{itemize}
Then
\[|\mathcal{A}||\mathcal{B}| \leq 4(1+O(1/n))((n-2)!)^{2}.\]
\end{lemma}
\begin{proof}
Let
\[J = \{j \in [n]:\ |\mathcal{A}_{1 \mapsto j}| \geq \frac{|\mathcal{A}|}{2n}\}.\]
Then \(|\mathcal{A} \setminus \mathcal{A}_{1 \to J}| < |\mathcal{A}|/2\), so \(|\mathcal{A}_{1 \to J}| > |\mathcal{A}|/2\).

Fix \(j \in [n]\), and consider the families
\[\mathcal{A}_{1 \mapsto j},\ \mathcal{B}_{1 \mapsto j}.\]
Notice that
\[(1\ j)\mathcal{A}_{1 \mapsto j},(1\ j)\mathcal{B}_{1 \mapsto j}\]
are families of permutations fixing 1 and with \(|\sigma \cap \tau| \neq 2\) for any \(\sigma \in (1\ j)\mathcal{A}_{1 \mapsto j}\) and any \(\tau \in (1\ j)\mathcal{B}_{1 \mapsto j}\). Restricting them to \(\{2,\ldots,n\}\) yields a pair of families \(\mathcal{E},\mathcal{F} \subset S_{\{2,\ldots,n\}}\) with \(|\sigma \cap \tau| \neq 1\ \forall \sigma \in \mathcal{E},\ \tau \in \mathcal{F}\). Applying Proposition \ref{prop:cross} to \(\mathcal{E}\) and \(\mathcal{F}\) gives
\[|\mathcal{E}||\mathcal{F}| \leq (1+O(1/n))((n-3)!)^{2},\]
so
\[|\mathcal{A}_{1 \mapsto j}||\mathcal{B}_{1 \mapsto j}| \leq (1+O(1/n))((n-3)!)^{2}.\]
Hence,
\[\frac{|\mathcal{A}|}{2n}|\mathcal{B}_{1 \mapsto j}| < (1+O(1/n))((n-3)!)^{2}\ \forall j \in J,\]
and therefore
\[|\mathcal{A}||\mathcal{B}_{1 \mapsto j}| < 2n(1+O(1/n))((n-3)!)^{2} \ \forall j \in J.\]
Summing over all \(j \in J\) gives:
\begin{equation}
\label{eq:crossbound1}
|\mathcal{A}||\mathcal{B}_{1 \to J}| < 2(1+O(1/n))((n-2)!)^{2}.
\end{equation}
Notice that no permutation in \(\mathcal{A}_{1 \to J}\) can have singleton intersection with any permutation in \(\mathcal{B}_{1 \nrightarrow J}\), and therefore by Proposition \ref{prop:cross} again, 
\[|\mathcal{A}_{1 \to J}||\mathcal{B}_{1 \nrightarrow J}| \leq (1+O(1/n))((n-2)!)^{2}.\]
Hence,
\begin{equation}
\label{eq:crossbound2}
|\mathcal{A}||\mathcal{B}_{1 \nrightarrow J}| < 2(1+O(1/n))((n-2)!)^{2}.
\end{equation}
Combining (\ref{eq:crossbound1}) and (\ref{eq:crossbound2}) gives:
\[|\mathcal{A}||\mathcal{B}| \leq 4(1+O(1/n))((n-2)!)^{2},\]
as required.
\end{proof}

\begin{lemma}
\label{lemma:3posvalues}
Let \(\mathcal{A},\mathcal{B} \subset S_{n}\) such that for any \(\sigma \in \mathcal{A}\) and any \(\tau \in \mathcal{B}\),
\begin{itemize}
\item If exactly one of \(\sigma(1)=\tau(1),\ \sigma(2)=\tau(2)\) holds, then \(|\sigma \cap \tau| \neq 2\);
\item If both hold, then \(|\sigma \cap \tau| \neq 3\);
\item If neither holds, then \(|\sigma \cap \tau| \neq 1\).
\end{itemize}
Then
\[|\mathcal{A}||\mathcal{B}| \leq 16(1+O(1/n))((n-2)!)^{2}.\]
\end{lemma}
\begin{proof}
Let
\[K = \{k \in [n]: |\mathcal{A}_{1 \mapsto k}| \geq \frac{|\mathcal{A}|}{2n}\}.\]
As in the proof of Lemma \ref{lemma:twoposvalues}, we have \(|\mathcal{A} \setminus \mathcal{A}_{1 \to K}| < |\mathcal{A}|/2\), so \(|\mathcal{A}_{1 \to K}| > |\mathcal{A}|/2\).

Fix \(k \in [n]\), and consider the pair of families
\[\mathcal{A}_{1 \mapsto k},\ \mathcal{B}_{1 \mapsto k}.\]
Notice that
\[(1\ k)\mathcal{A}_{1 \mapsto k},\ (1\ k)\mathcal{B}_{1 \mapsto k}\]
are families of permutations fixing 1, such that for any \(\sigma \in (1\ j)\mathcal{A}_{1 \mapsto j},\) and any \(\tau \in (1\ j)\mathcal{B}_{1 \mapsto j}\),
\begin{itemize}
\item if \(\sigma(2) = \tau(2)\) we have \(|\sigma \cap \tau|\neq 3\);
\item if \(\sigma(2)\neq \tau(2)\) we have \(|\sigma \cap \tau| \neq 2\).
\end{itemize}
Restricting this pair of families to \(\{2,\ldots,n\}\) yields a pair of families \(\mathcal{E},\mathcal{F} \subset S_{\{2,\ldots,n\}}\) such that for any \(\sigma \in \mathcal{E}\) and any \(\tau \in \mathcal{F}\),
\begin{itemize}
\item if \(\sigma(2) = \tau(2)\) we have \(|\sigma \cap \tau|\neq 2\);
\item if \(\sigma(2)\neq \tau(2)\) we have \(|\sigma \cap \tau| \neq 1\).
\end{itemize}
Applying Lemma \ref{lemma:twoposvalues} to \(\mathcal{E}\) and \(\mathcal{F}\) gives
\[|\mathcal{E}||\mathcal{F}| \leq 4(1+O(1/n))((n-3)!)^{2},\]
so
\[|\mathcal{A}_{1 \mapsto k}||\mathcal{B}_{1 \mapsto k}| \leq 4(1+O(1/n))((n-3)!)^{2}.\]
Hence,
\[\frac{|\mathcal{A}|}{2n}|\mathcal{B}_{1 \mapsto k}| < 4(1+O(1/n))((n-3)!)^{2}\quad \forall k \in K,\]
i.e.
\[|\mathcal{A}||\mathcal{B}_{1 \mapsto k}| < 8n(1+O(1/n))((n-3)!)^{2}\quad \forall k \in K.\]
Summing over all \(k \in K\) gives
\begin{equation}
\label{eq:crossbound3}
|\mathcal{A}||\mathcal{B}_{1 \to K}| < 8(1+O(1/n))((n-2)!)^{2}.
\end{equation}
Observe that \(\mathcal{A}_{1 \to K},\ \mathcal{B}_{1 \nrightarrow K}\) are families of permutations satisfying the hypotheses of Lemma \ref{lemma:twoposvalues}, and therefore
\[|\mathcal{A}_{1 \to K}||\mathcal{B}_{1 \nrightarrow K}| \leq 4(1+O(1/n))^{2}((n-2)!)^{2}.\]
Hence,
\begin{equation}
\label{eq:crossbound4}
|\mathcal{A}||\mathcal{B}_{1 \nrightarrow K}| \leq 8(1+O(1/n))^{2}((n-2)!)^{2}
\end{equation}
Combining equations (\ref{eq:crossbound3}) and (\ref{eq:crossbound4}) gives
\[|\mathcal{A}||\mathcal{B}| \leq 16(1+O(1/n))((n-2)!)^{2}\]
as required.
\end{proof}

We can now prove Theorem \ref{thm:roughstability}.

\begin{proof}[Proof of Theorem \ref{thm:roughstability}.]
Let \(\mathcal{A} \subset S_{n}\) with no singleton intersection, and with \(|\mathcal{A}| \geq c(n-2)!\).

Note that by taking $K_c \geq 1/c^{26}$, we may assume that $c \geq n^{-1/13}$ (otherwise the conclusion of the theorem holds trivially). Note also that by taking $K_c \geq n_0(c)^2$, we may assume throughout that $n \geq n_0(c)$ for any integer $n_0(c)$ depending only on $c$.

We first apply Lemma \ref{lemma:stability-hoffman} to our graph $\Gamma=\Gamma_n$, with $X = \mathcal{A}$ and $M = n! - (f^{(n-2,2)})^2 - (f^{(n-2,1^2)})^2$, so that
$$\{\lambda_{M+1},\lambda_{M+2} \ldots, \lambda_N\} = \{\lambda_{(n-2,2)}, \lambda_{(n-2,1^2)}\},$$
and
$$U = \textrm{Span}\{\mathbf{f}\} \oplus U_{(n-2,2)} \oplus U_{(n-2,1^2)}.$$
By (\ref{eq:uniformbound}) and the table following it, we have
$$|\lambda_M| \leq C(n-3)! \leq \frac{C'}{n}|\lambda_N|,$$
where $C'$ is an absolute constant. Hence, by Lemma \ref{lemma:stability-hoffman}, we have
\[||\chi_{\mathcal{A}} - P_{U}(\chi_{\mathcal{A}})||_2^{2} \leq (1-c)(1+O(1/n))||\chi_{\mathcal{A}}||_2^{2}.\]

Now let
\[U_{t} = \bigoplus_{\alpha \vdash n:\ \alpha \textrm{ is }t\textrm{-fat}} U_{\alpha}.\]
Since the partitions $(n)$, $(n-2,2)$ and $(n-2,1^2)$ are all 2-fat, we have \(U \leq U_{2}\), so
\[||\chi_{\mathcal{A}} - P_{U_{2}}(\chi_{\mathcal{A}})||_2^{2} \leq ||\chi_{\mathcal{A}} - P_{U}(\chi_{\mathcal{A}})||_2^{2} \leq (1-c)(1+O(1/n))||\chi_{\mathcal{A}}||_2^{2}.\]

It follows from the proof of Lemma 7 in \cite{dezafranklstability} that for any fixed \(c > 0\) and \(t \in \mathbb{N}\), if
\[||\chi_{\mathcal{A}} - P_{U_{t}}(\chi_{\mathcal{A}})||_2^{2} \leq (1-c)(1+O(1/n))||\chi_{\mathcal{A}}||_2^{2},\]
then there exist \(i_{1},\ldots,i_{t}\) and \(j_{1},\ldots,j_{t}\) such that
\[|A_{i_{1}\mapsto j_{1},\ldots,i_{t} \mapsto j_{t}}| \geq f(n) c(n-2t)!,\]
where \(f(n) = \Theta(\sqrt{n/\log n})\), provided $c = \Omega(n^{-1/12})$ (which holds by assumption).

Applying this with \(t=2\) shows that there exist \(i,j,k\) and \(l \in [n]\) such that
\[|\mathcal{A}_{i \mapsto j,k \mapsto l}| \geq f(n) c(n-4)!.\]

Without loss of generality, we may assume \(i=j=1,k=l=2\). Hence, we have
\begin{equation}
\label{eq:lowerbound1122}
|\mathcal{A}_{1 \mapsto 1,2 \mapsto 2}| \geq f(n)c(n-4)!.
\end{equation}

Our aim is now to show that all but at most \(o((n-2)!)\) of the permutations in \(\mathcal{A}\) fix either \(1\) or \(2\). To show this, it will suffice to prove the following.

\begin{claim}
If \(j \neq 1\) and \(l\neq 2\), then \(|\mathcal{A}_{1 \mapsto j,2 \mapsto l}| \leq o((n-4)!)\).
\end{claim}

\begin{proof}[Proof of claim.]
First, we deal with the case \(j=2,\ l=1\). Clearly, restricting the families
\[\mathcal{A}_{1 \mapsto 2,2 \mapsto 1},\ \mathcal{A}_{1 \mapsto 1,2 \mapsto 2}\]
to \(\{3,4,\ldots,n\}\) (`deleting 1 and 2') produces two families \(\mathcal{C},\mathcal{D} \subset S_{\{3,\ldots,n\}}\) such that $|\sigma \cap \tau| \neq 1\ \forall \sigma \in \mathcal{C},\tau \in \mathcal{D}$. Applying Proposition \ref{prop:cross} to $\mathcal{C},\mathcal{D}$ gives
\[|\mathcal{A}_{1 \mapsto 2,2 \mapsto 1}||\mathcal{A}_{1 \mapsto 1,2 \mapsto 2}| = |\mathcal{C}||\mathcal{D}| \leq (1+O(1/n))((n-4)!)^{2}.\]
Therefore, by (\ref{eq:lowerbound1122}),
\[|\mathcal{A}_{1 \mapsto 2,2\mapsto 1}| \leq o((n-4)!),\]
as required.

Next, we deal with the case $j=2,\ l \notin \{1,2\}$. Without loss of generality, we may assume that \(l=3\); we need to show that \(|\mathcal{A}_{1 \mapsto 2,2 \mapsto 3}| \leq o((n-4)!)\).

Consider the pair of families
\[\mathcal{A}_{1 \mapsto 1,2 \mapsto 2},\ (1\ 3\ 2)\mathcal{A}_{1 \mapsto 2,2 \mapsto 3}.\]
Let \(\sigma \in \mathcal{A}_{1 \mapsto 1,2 \mapsto 2},\ \tau' = (1 \ 3\ 2)\tau\) where \(\tau \in \mathcal{A}_{1 \mapsto 2, 2 \mapsto 3}\). Let \(a = \sigma^{-1}(3)\). Observe that
\begin{itemize}
\item If \(\tau(a)=1\), i.e. \(\tau'^{-1}(3) = a\), then \(\sigma(a) = \tau'(a) = 3\), so \(\sigma\) and \(\tau'\) agree wherever \(\sigma\) and \(\tau\) agree, and also at \(1,2\) and \(a\), but nowhere else. Hence, \(\sigma\) and \(\tau'\) cannot agree at exactly 4 points.
\item If \(\tau(a) \neq 1\), i.e. \(\tau'^{-1}(3) \neq a\), then \(\sigma(a) \neq \tau'(a)\), and therefore \(\sigma\) and \(\tau'\) agree wherever \(\sigma\) and \(\tau\) agree, and also at \(1\) and \(2\), but nowhere else. Hence, \(\sigma\) and \(\tau'\) cannot agree at exactly 3 points.
\end{itemize}
Let
\[\mathcal{C} = (\mathcal{A}_{1 \mapsto 1,2 \mapsto 2})^{-1},\quad \mathcal{D} = ((1 \ 3 \ 2)\mathcal{A}_{1 \mapsto 2,2 \mapsto 3})^{-1}.\]
Then \(\mathcal{C}\) and \(\mathcal{D}\) are families of permutations fixing both \(1\) and \(2\). Moreover, for any \(\rho \in \mathcal{C}\) and any \(\pi \in \mathcal{D}\),
\begin{itemize}
\item if \(\rho(3)=\pi(3)\), then \(|\rho \cap \pi| \neq 4\);
\item if \(\rho(3) \neq \pi(3)\), then \(|\rho \cap \pi| \neq 3\).
\end{itemize}
Restrict \(\mathcal{C}\) and \(\mathcal{D}\) to \(\{3,4,\ldots,n\}\) (`delete 1 and 2') to obtain the families \(\mathcal{C}',\mathcal{D}' \subset S_{\{3,\ldots,n\}}\). Observe that for any \(\rho' \in \mathcal{C}'\) and any \(\pi' \in \mathcal{D}'\),
\begin{itemize}
\item if \(\rho'(3)=\pi'(3)\), then \(|\rho' \cap \pi'| \neq 2\);
\item if \(\rho'(3) \neq \pi'(3)\), then \(|\rho' \cap \pi'| \neq 1\).
\end{itemize} 
Clearly, \(|\mathcal{C}'| = |\mathcal{A}_{1 \mapsto 1,2 \mapsto 2}|\) and \(|\mathcal{D}'| = |\mathcal{A}_{1 \mapsto 2,2 \mapsto 3}|\). Applying Lemma \ref{lemma:twoposvalues} to \(\mathcal{C}'\) and \(\mathcal{D}'\) (with ground set \(\{3,\ldots,n\}\)) shows that
\[|\mathcal{A}_{1 \mapsto 1,2 \mapsto 2}||\mathcal{A}_{1 \mapsto 2,2\mapsto 3}| = |\mathcal{C}'||\mathcal{D}'| \leq 4(1+O(1/n))^{2} ((n-4)!)^{2}.\]
Therefore, by (\ref{eq:lowerbound1122}),
\[|\mathcal{A}_{1 \mapsto 2,2 \mapsto 3}| \leq o((n-4)!)\]
as required.

The case $l=1,\ j \notin \{1,2\}$ is the same as the previous case. It remains to deal with the case \(\{j,l\} \cap \{1,2\} = \emptyset\). Without loss of generality, we may assume that \(j=3\) and \(l=4\); we just need to show that \(|\mathcal{A}_{1 \mapsto 3,2 \mapsto 4}| \leq o((n-4)!)\).

Consider the pair of families
\[\mathcal{A}_{1 \mapsto 1,2 \mapsto 2},\ (1\ 3)(2 \ 4)\mathcal{A}_{1 \mapsto 3,2 \mapsto 4}.\]
Let \(\sigma \in \mathcal{A}_{1 \mapsto 1,2 \mapsto 2},\tau' = (1 \ 3)(2 \ 4)\tau\) where \(\tau \in \mathcal{A}_{1 \mapsto 3, 2 \mapsto 4}\). Let $a = \sigma^{-1}(3)$, and let $b=\sigma^{-1}(4)$.

If \(\tau(a)=1\), i.e. \(\tau'^{-1}(3) = a\), then \(\sigma(a) = \tau'(a) = 3\). Similarly, if \(\tau(b) = 2\), i.e. \(\tau'^{-1}(4)=b\), then \(\sigma(b) = \tau'(b)=4\). Hence, \(\sigma\) and \(\tau'\) agree wherever \(\sigma\) and \(\tau\) agree, and also at \(1\) and \(2\), and possibly at \(a\) or \(b\), but nowhere else.

Let
\[\mathcal{C} = (\mathcal{A}_{1 \mapsto 1,2 \mapsto 2})^{-1},\ \mathcal{D} = ((1 \ 3)(2 \ 4)\mathcal{A}_{1 \mapsto 3,2 \mapsto 4})^{-1}.\]
Then \(\mathcal{C}\) and \(\mathcal{D}\) are families of permutations fixing both \(1\) and \(2\). Moreover, for any \(\rho \in \mathcal{C}\) and any \(\pi \in \mathcal{D}\),
\begin{itemize}
\item if exactly one of \(\rho(3)=\pi(3)\) and \(\rho(4)=\pi(4)\) holds, then \(|\rho \cap \pi| \neq 4\);
\item if both hold, then \(|\rho \cap \pi| \neq 5\);
\item if neither hold, then \(|\rho \cap \pi| \neq 3\).
\end{itemize}
Restrict \(\mathcal{C}\) and \(\mathcal{D}\) to \(\{3,4,\ldots,n\}\) (`delete 1 and 2') to obtain the families \(\mathcal{C}',\mathcal{D}' \subset S_{\{3,\ldots,n\}}\); note that for any \(\rho' \in \mathcal{C}'\) and any \(\pi' \in \mathcal{D}'\),
\begin{itemize}
\item if exactly one of \(\rho'(3)=\pi'(3)\) and \(\rho'(4)=\pi'(4)\) holds, then \(|\rho' \cap \pi'| \neq 2\);
\item if both hold, then \(|\rho' \cap \pi'| \neq 3\);
\item if neither hold, then \(|\rho' \cap \pi'| \neq 1\).
\end{itemize}
Clearly, \(|\mathcal{C}'| = |\mathcal{A}_{1 \mapsto 1,2 \mapsto 2}|\) and \(|\mathcal{D}'| = |\mathcal{A}_{1 \mapsto 3,2 \mapsto 4}|\).

Applying Lemma \ref{lemma:3posvalues} to \(\mathcal{C}'\) and \(\mathcal{D}'\) (with ground set \(\{3,4,\ldots,n\}\)) shows that
\[|\mathcal{A}_{1 \mapsto 1,2 \mapsto 2}||\mathcal{A}_{1 \mapsto 3,2\mapsto 4}| = |\mathcal{C}'||\mathcal{D}'| \leq 4(1+O(1/n))^{2} ((n-4)!)^{2}.\]
Therefore, by (\ref{eq:lowerbound1122}),
\[|\mathcal{A}_{1 \mapsto 3,2 \mapsto 4}| \leq o((n-4)!)\]
as required.
\end{proof}

Summing the inequality in the above claim over all \((n-1)^{2}\) possible pairs $(j,l)$ gives:
\[|\mathcal{A}\setminus (\mathcal{A}_{1 \mapsto 1} \cup \mathcal{A}_{2 \mapsto 2})| = \sum_{(j,l):\ j \neq 1, l \neq 2} |\mathcal{A}_{1 \mapsto j,2 \mapsto l}| \leq o((n-2))!.\]

It follows that \(|\mathcal{A}_{1 \mapsto 1} \cup \mathcal{A}_{2 \mapsto 2}| \geq c(n-2)!-o((n-2)!)\), and therefore
\[\max(|\mathcal{A}_{1 \mapsto 1}|,|\mathcal{A}_{2 \mapsto 2}|) \geq \tfrac{1}{2}(1-o(1))c(n-2)! > c'(n-2)!,\]
where \(c' := c/4\), provided $n$ is sufficiently large depending on $c$. Without loss of generality, we may assume that
\begin{equation}
\label{eq:lowerbound1}
|\mathcal{A}_{1 \mapsto 1}| \geq c'(n-2)!.
\end{equation}

Our next aim is to show that all but at most \(O((n-3)!)\) of the permutations in \(\mathcal{A}\) fix 1. To show this, it suffices to prove the following.

\begin{claim}
For each $j \neq 1$, we have \(|\mathcal{A}_{1 \mapsto j}| \leq O(1/c)(n-5)!\).
\end{claim}
\begin{proof}[Proof of claim.]
Fix \(j \neq 1\), and consider the pair of families
\[\mathcal{A}_{1 \mapsto 1},\quad (1\ j)\mathcal{A}_{1 \mapsto j}.\]
Let \(\sigma \in \mathcal{A}_{1 \mapsto 1},\tau' = (1 \ j)\tau\) where \(\tau \in \mathcal{A}_{1 \mapsto j}\). Let \(a = \sigma^{-1}(j)\). If \(\tau(a)=1\), i.e. \(\tau'^{-1}(j) = a\), then \(\sigma(a) = \tau'(a) = j\), so \(\sigma\) and \(\tau'\) agree wherever \(\sigma\) and \(\tau\) agree, and also at \(1\) and \(a\), but nowhere else. Hence, \(\sigma\) and \(\tau'\) cannot agree at exactly 3 points. If \(\tau(a) \neq 1\), i.e. \(\tau'^{-1}(j) \neq a\), then \(\sigma(a) \neq \tau'(a)\), so \(\sigma\) and \(\tau'\) agree wherever \(\sigma\) and \(\tau\) agree, and also at \(1\), but nowhere else. Hence, \(\sigma\) and \(\tau'\) cannot agree at exactly 2 points.

Let
\[\mathcal{C} = (\mathcal{A}_{1 \mapsto 1})^{-1},\ \mathcal{D} = ((1 \ j)\mathcal{A}_{1 \mapsto j})^{-1}.\]
Then \(\mathcal{C}\) and \(\mathcal{D}\) are families of permutations fixing \(1\). Moreover, for any \(\rho \in \mathcal{C}\) and any \(\pi \in \mathcal{D}\),
\begin{itemize}
\item if \(\rho(j)=\pi(j)\), then \(|\rho \cap \pi| \neq 3\);
\item if \(\rho(j) \neq \pi(j)\), then \(|\rho \cap \pi| \neq 2\).
\end{itemize}
Restrict \(\mathcal{C}\) and \(\mathcal{D}\) to \(\{2,3,\ldots,n\}\) (`delete 1') to obtain the families \(\mathcal{C}',\mathcal{D}' \subset S_{\{2,\ldots,n\}}\). Observe that for any \(\rho' \in \mathcal{C}'\) and any \(\pi' \in \mathcal{D}'\),
\begin{itemize}
\item if \(\rho'(j)=\pi'(j)\), then \(|\rho' \cap \pi'| \neq 2\);
\item if \(\rho'(j) \neq \pi'(j)\), then \(|\rho' \cap \pi'| \neq 1\).
\end{itemize}

Clearly, \(|\mathcal{C}'| = |\mathcal{A}_{1 \mapsto 1,2 \mapsto 2}|\) and \(|\mathcal{D}'| = |\mathcal{A}_{1 \mapsto 2,2 \mapsto 3}|\). Applying Lemma \ref{lemma:twoposvalues} to \(\mathcal{C}',\mathcal{D}'\) yields
\[|\mathcal{C}'||\mathcal{D}'|\leq 4(1+O(1/n))((n-3)!)^{2},\]
and therefore
\[|\mathcal{A}_{1 \mapsto 1}||\mathcal{A}_{1 \mapsto j}| \leq 4(1+O(1/n))((n-3)!)^{2}.\]
Hence, by (\ref{eq:lowerbound1}),
\[|\mathcal{A}_{1 \mapsto j}| \leq \frac{4(1+O(1/n))((n-3)!)^{2}}{c'(n-2)!} = O(1/c)(n-5)!,\]
proving the claim.
\end{proof}

Summing the inequality of the above claim over all \(j \neq 1\) gives
\begin{equation}\label{eq:almost} |\mathcal{A} \setminus \mathcal{A}_{1 \mapsto 1}| = \sum_{j \neq 1} |\mathcal{A}_{1 \mapsto j}| \leq O(1/c)(n-4)!.\end{equation}

Observe that \(\mathcal{A}_{1 \mapsto 1} \subset S_n\) is a 2-intersecting family of permutations all fixing 1. Restricting $\mathcal{A}_{1 \mapsto 1}$ to \(\{2,\ldots,n\}\) yields a 1-intersecting family of permutations \(\mathcal{A}_{1 \mapsto 1}' \subset S_{\{2,\ldots,n\}}\), with
$$|\mathcal{A}_{1 \mapsto 1}'| \geq c'(n-2)!.$$

Recall the following stability theorem for 1-intersecting families of permutations, proved by the author in \cite{cameronkuconj}.

\begin{theorem}
\label{thm:roughstabilityintersecting}
Let \(c > 0\) be a positive constant. If \(\mathcal{A}\subset S_{n}\) is a 1-intersecting family of permutations with \(|\mathcal{A}|\geq c(n-1)!\), then there exist \(i,j \in [n]\) such that all but at most \(O_{c}((n-2)!)\) permutations in \(\mathcal{A}\) map \(i\) to \(j\).
\end{theorem}

Applying this theorem to \(\mathcal{A}_{1\mapsto 1}'\) (with ground set \(\{2,\ldots,n\}\)), we see that there exist \(i,j \geq 2\) such that all but at most \(O_{c}((n-3)!)\) permutations in \(\mathcal{A}_{1 \mapsto 1}\) map \(i\) to \(j\). Without loss of generality, we may assume that \(i=j=2\), so
$$|\mathcal{A}_{1 \mapsto 1} \setminus \mathcal{A}_{1 \mapsto 1,2\mapsto 2}| \leq O_c((n-3)!).$$
Combining this with (\ref{eq:almost}) yields
$$|\mathcal{A} \setminus \mathcal{A}_{1\mapsto 1,2\mapsto 2}| \leq O_c((n-3)!),$$
completing the proof of Theorem \ref{thm:roughstability}.
\end{proof}

We immediately obtain the following corollary.
\begin{corollary}
\label{corr:stab}
For any \(c > 1-1/e\), and any \(n\) sufficiently large depending on \(c\), if \(\mathcal{A} \subset S_{n}\) is a family of permutations with no singleton intersection and with \(|\mathcal{A}| \geq c(n-2)!\), then there exist \(i,j,k\) and \(l\) such that every permutation in \(\mathcal{A}\) maps \(i\) to \(j\) and \(k\) to \(l\), i.e. \(\mathcal{A}\) is contained within a 2-coset of \(S_{n}\).
\end{corollary}
\begin{proof}
By Theorem \ref{thm:roughstability}, there exist \(i,j,k\) and \(l\) such that
\[|\mathcal{A} \setminus \mathcal{A}_{i \mapsto j,k \mapsto l}| \leq O((n-3)!).\]
Without loss of generality, \(i=j=1\) and \(k=l=2\), so
\[|\mathcal{A} \setminus \mathcal{A}_{1 \mapsto 1,2 \mapsto 2}| \leq O((n-3)!).\]
Suppose for a contradiction that \(\mathcal{A}\) contains a permutation \(\tau\) not fixing both \(1\) and \(2\). It is easy to check that the number of permutations fixing both \(1\) and \(2\) and agreeing with \(\tau\) at exactly one point is
\[(1/e+o(1))(n-2)!.\]
No such permutation can be in \(\mathcal{A}\), so
\[|\mathcal{A}| = |\mathcal{A}_{1 \mapsto 1,2 \mapsto 2}|+|\mathcal{A} \setminus \mathcal{A}_{1 \mapsto 1,2 \mapsto 2}| \leq (1-1/e+o(1))(n-2)!+O((n-3)!),\]
a contradiction provided \(n\) is sufficiently large depending on \(c\).
\end{proof}

Corollary \ref{corr:stab} clearly implies our main theorem, Theorem \ref{thm:main}.

\section{A Hilton-Milner type result}
In this section, we will use Theorem \ref{thm:roughstability} to prove the following Hilton-Milner type result.
\begin{theorem}
\label{thm:hm}
For \(n\) sufficiently large, if \(\mathcal{A} \subset S_{n}\) is a family of permutations with no singleton intersection, and \(\mathcal{A}\) is not contained in a 2-coset of \(S_{n}\), then \(|\mathcal{A}| \leq |\mathcal{B}|\), where
\begin{align*}
\mathcal{B} = & \{\sigma \in S_{n}:\ \sigma(1)=1,\sigma(2)=2,\ \#\{\textrm{fixed points of }\sigma \geq 5\} \neq 1\}\\
& \cup \{(1\ 3)(2\ 4),(1\ 4)(2\ 3),(1\ 3\ 2\ 4),(1\ 4\ 2\ 3)\}.
\end{align*}
Equality holds only if \(\mathcal{A}\) is a double translate of \(\mathcal{B}\), meaning that there exist \(\pi,\tau \in S_{n}\) such that \(\mathcal{A}=\pi \mathcal{B} \tau\).
\end{theorem}

\begin{proof}
Let \(\mathcal{A}\) be a family of maximum size satisfying the conditions of Theorem \ref{thm:hm}. Observe that \(\mathcal{B}\) satisfies them, and has size
\[(n-2)!-(n-4)(d_{n-3}+2d_{n-4}+d_{n-5})+4 = (1-1/e+o(1))(n-2)!.\]

To see this, note that the number of permutations fixing 1 and 2 and with exactly one fixed point \(\geq 5\) is
\[(n-4)(d_{n-3}+2d_{n-4}+d_{n-5}).\]
Indeed, there are \(n-4\) choices of a number \(p \geq 5\) to fix, and \(d_{n-3}+2d_{n-4}+d_{n-5}\) permutations whose fixed points are 1, 2, \(p\) and the numbers in some subset of \(\{3,4\}\).

Therefore, by our assumption of the maximality of \(|\mathcal{A}|\),
\begin{equation}
\label{eq:lowerboundA}
|\mathcal{A}| \geq |\mathcal{B}| = (1-1/e+o(1))(n-2)!.
\end{equation}
By Theorem \ref{thm:roughstability}, there exist \(i,j,k\) and \(l\) such that
\[|\mathcal{A} \setminus \mathcal{A}_{i \mapsto j,k \mapsto l}| \leq O((n-3)!).\]
By double translation, we may assume that \(i=j=1\) and \(k=l=2\), so
\[|\mathcal{A} \setminus \mathcal{A}_{1 \mapsto 1,2 \mapsto 2}| \leq O((n-3)!).\]
By assumption, \(\mathcal{A}\) must contain some permutation \(\tau\) not fixing both 1 and 2. There are four possibilities:
\begin{enumerate}
\item \(\tau(1)=2,\ \tau(2)=1\);
\item \(\tau\) fixes 1 or 2;
\item \(\tau\) does not fix 1 or 2, but \(|\tau\{1,2\} \cap \{1,2\}|=1\);
\item \(\tau\{1,2\} \cap \{1,2\} = \emptyset\);
\end{enumerate}

First, we obtain bounds on \(|\mathcal{A}_{1\mapsto 1,2 \mapsto 2}|\) in each of the cases 1-3.

In case 1, by considering the translated family \(\mathcal{A}\tau^{-1}(1\ 2)\), we may assume that \(\tau = (1\ 2)\). Hence, \(\mathcal{A}_{1\mapsto 1,2\mapsto 2}\) cannot contain any permutation in the family
\[\mathcal{F}_{1} = \{\sigma \in S_{n}:\ \sigma(1)=1,\ \sigma(2)=2,\ \#\{\textrm{fixed points of }\sigma \geq 3\} = 1\},\]
since every \(\sigma \in \mathcal{F}_{1}\) has \(|\sigma \cap (1\ 2)|=1\). Observe that
\[|\mathcal{F}_{1}|=(n-2)d_{n-3},\]
the number of permutations of \(\{3,4,\ldots,n\}\) with exactly one fixed point. Let
\[\mathcal{G}_{1} = \{\sigma \in S_{n}: \sigma(1)=1,\ \sigma(2)=2\} \setminus \mathcal{F}_{1};\]
then
\[\mathcal{A}_{1 \mapsto 1,2 \mapsto 2} \subset \mathcal{G}_{1},\]
and therefore
\[|\mathcal{A}_{1\mapsto 1,2\mapsto 2}| \leq |\mathcal{G}_1| = (n-2)!-(n-2)d_{n-3}.\]

In case 2, we may assume that \(\tau(2)=2,\ \tau(1)=3\). Then, by considering the translated family \( \mathcal{A}\tau^{-1}(1\ 3)\), we may assume that \(\tau = (1\ 3)\). Hence, \(\mathcal{A}_{1\mapsto 1,2\mapsto 2}\) cannot contain any permutation in the family
\[\mathcal{F}_{2} = \{\sigma \in S_{n}:\ \sigma(1)=1,\ \sigma(2)=2,\ \sigma\textrm{ has no fixed points}\geq 4\},\]
since every \(\sigma \in \mathcal{F}_{2}\) has \(|\sigma \cap (1\ 3)|=1\). Observe that
\[|\mathcal{F}_{2}|=d_{n-2}+d_{n-3},\]
since there are \(d_{n-2}\) permutations whose fixed points are 1 and 2, and \(d_{n-3}\) permutations whose fixed points are 1, 2 and 3. Let
\[\mathcal{G}_{2} = \{\sigma \in S_{n}: \sigma(1)=1,\ \sigma(2)=2\} \setminus \mathcal{F}_{2};\]
then
\[\mathcal{A}_{1 \mapsto 1,2 \mapsto 2} \subset \mathcal{G}_{2},\]
and therefore
\[|\mathcal{A}_{1\mapsto 1,2\mapsto 2}| \leq |\mathcal{G}_{2}| = (n-2)!-d_{n-2}-d_{n-3}.\]

In case 3, we may assume that \(\tau(1)=2\) and \(\tau(2)=3\). Then, by considering the translated family \(\mathcal{A}\tau^{-1}(1\ 2\ 3)\), we may assume that \(\tau = (1\ 2\ 3)\). Hence, \(\mathcal{A}_{1\mapsto 1,2\mapsto 2}\) cannot contain any permutation in the family
\[\mathcal{F}_{3} = \{\sigma \in S_{n}:\ \sigma(1)=1,\ \sigma(2)=2,\ \sigma\textrm{ has exactly one fixed point}\geq 4\},\]
since every \(\sigma \in \mathcal{F}_{3}\) has \(|\sigma \cap (1\ 2\ 3)|=1\). Observe that
\[|\mathcal{F}_{3}|=(n-3)(d_{n-3}+d_{n-4}),\]
since there are \(n-3\) choices of a point \(i \geq 4\) to fix, \(d_{n-3}\) permutations whose fixed points are 1, 2 and \(i\), and \(d_{n-4}\) permutations whose fixed points are 1, 2, 3 and \(i\). Let
\[\mathcal{G}_{3} = \{\sigma \in S_{n}: \sigma(1)=1,\ \sigma(2)=2\} \setminus \mathcal{F}_{3};\]
then
\[\mathcal{A}_{1 \mapsto 1,2 \mapsto 2} \subset \mathcal{G}_{3},\]
and therefore
\[|\mathcal{A}_{1\mapsto 1,2\mapsto 2}| \leq |\mathcal{G}_{3}| = (n-2)!-(n-3)(d_{n-3}+d_{n-4}).\]

It is easily checked that for each \(j \leq 3\),
\begin{equation}\label{eq:boundG} |\mathcal{G}_{j}| \leq |\mathcal{B}|-4-d_{n-4} = (1-1/e+o(1))(n-2)!.\end{equation}
Since in case $j$, we have shown that \(\mathcal{A} \subset \mathcal{G}_{j}\) (for $j=1,2,3$), it follows from (\ref{eq:lowerboundA}) that
\begin{equation}
\label{eq:toomanyothers}
|\mathcal{A} \setminus \mathcal{A}_{1 \mapsto 1,2 \mapsto 2}| \geq d_{n-4}+4 = (1/e+o(1))(n-4)!.
\end{equation}
We now show that this leads to a contradiction. The number of permutations in \(S_{n}\) with at least \(\lfloor n/2 \rfloor\) fixed points is at most
\[{n \choose \lfloor n/2 \rfloor} (\lceil n/2 \rceil)! = n!/(\lfloor n / 2 \rfloor)! = o((n-4)!),\]
so the number of permutations in \(\mathcal{A} \setminus \mathcal{A}_{1 \mapsto 1,2 \mapsto 2}\) with less than \(\lfloor n / 2 \rfloor\) fixed points is at least \((1/e-o(1))(n-4)!\). Choose one such permutation \(\pi \in \mathcal{A} \setminus \mathcal{A}_{1 \mapsto 1,2 \mapsto 2}\). We now prove the following.

\begin{claim}
\label{claim:H}
In each of the cases 1-3, if there exists $\pi \in \mathcal{A}\setminus \mathcal{A}_{1 \mapsto 1,2 \mapsto 2}$ with at most $\lfloor n/2 \rfloor$ fixed points, then
$$|\mathcal{A}_{1\mapsto 1,2 \mapsto 2}| \leq (1-1/e - 1/(8e)+o(1))(n-2)!.$$
\end{claim}
\begin{proof}[Proof of claim.]
Let
\[N(\pi) = \{i \geq 5:\ \pi(i) \neq i\};\]
then \(|N(\pi)| \geq n-4-\lfloor n /2 \rfloor = \lceil n/2 \rceil -4.\)

Let
\[\mathcal{H} = \{\sigma \in S_{n}:\ \sigma \textrm{ fixes 1, 2, and at least two points in } N(\pi),\ \textrm{and}\ |\sigma \cap \pi|=1\}.\]
Observe that \(\mathcal{H} \subset \mathcal{G}_{j}\), since each \(\sigma \in \mathcal{H}\) has at least two fixed points \(\geq 4\), and that \(\mathcal{A} \cap \mathcal{H} = \emptyset\). If \(\pi\) fixes 1 or 2, then \(\mathcal{H}\) is the set of permutations fixing 1,2, and at least two points in \(N(\pi)\), and disagreeing with \(\pi\) at every point \(\geq 3\), so has size  
\[|\mathcal{H}| \geq {|N(\pi)|\choose 2} d_{n-4} \geq \tfrac{1}{8}(1/e-o(1))(n-2)!.\]
If \(\pi\) does not fix 1 or 2, then \(\mathcal{H}\) is the set of permutations fixing 1,2 and at least two points in \(N(\pi)\), and agreeing with \(\pi\) at exactly one point \(\geq 3\), so we have
\[|\mathcal{H}| \geq {|N(\pi)|\choose 2} (n-6)d_{n-5} \geq \tfrac{1}{8}(1/e-o(1))(n-2)!.\]
Since $\mathcal{A}_{1 \mapsto 1,2\mapsto 2} \subset \mathcal{G}_j$, it follows that
\begin{align*} |\mathcal{A}_{1 \mapsto 1,2\mapsto 2}| & \leq |\mathcal{G}_j| - |\mathcal{H}|\\
& \leq (1-1/e+o(1))(n-2)! - \tfrac{1}{8}(1/e - o(1))(n-2)!\\
& = (1-1/e-1/(8e)+o(1))(n-2)!,\end{align*}
as required.
\end{proof}

Since \(|\mathcal{A} \setminus \mathcal{A}_{1 \mapsto 1,2 \mapsto 2}| \leq O((n-3)!)\), it follows from the above claim that
\begin{align*} |\mathcal{A}| & \leq (1-1/e-1/(8e)+o(1))(n-2)!+O((n-3)!)\\
& = (1-1/e-1/(8e)+o(1))(n-2)!,\end{align*}
contradicting (\ref{eq:lowerboundA}).

Hence, we may assume that none of the cases 1-3 occur, so case 4 must occur, i.e.
\begin{equation}
\label{eq:12condition}
\tau\{1,2\} \cap \{1,2\} = \emptyset\quad \forall \tau \in \mathcal{A} \setminus \mathcal{A}_{1 \mapsto 1,2 \mapsto 2}.
\end{equation}
Take any \(\tau_0 \in \mathcal{A} \setminus \mathcal{A}_{1 \mapsto 1,2 \mapsto 2}\). Without loss of generality, we may assume that \(\tau_0(1)=3\) and \(\tau_0(2)=4\). Then, by considering the translated family \( \mathcal{A}\tau_0^{-1}(1\ 3)(2 \ 4)\), we may assume that \(\tau_0 = (1\ 3)(2\ 4)\). Hence, \(\mathcal{A}_{1\mapsto 1,2\mapsto 2}\) cannot contain any permutation in the family
\[\mathcal{F}_{4} = \{\sigma \in S_{n}:\ \sigma(1)=1,\ \sigma(2)=2,\ \#\{\textrm{fixed points of }\sigma \geq 5\} = 1\},\]
since every \(\sigma \in \mathcal{F}_{4}\) has \(|\sigma \cap (1\ 3)(2\ 4)|=1\). As observed above,
\[|\mathcal{F}_{4}|=(n-4)(d_{n-3}+2d_{n-4}+d_{n-5}).\]
Let
\[\mathcal{G}_{4} = \{\sigma \in S_{n}: \sigma(1)=1,\ \sigma(2)=2\} \setminus \mathcal{F}_{4};\]
then
\begin{equation}
\label{eq:containmentfor1122}
\mathcal{A}_{1 \mapsto 1,2 \mapsto 2} \subset \mathcal{G}_{4},
\end{equation}
and therefore
\begin{align*} |\mathcal{A}_{1\mapsto 1,2\mapsto 2}| & \leq |\mathcal{G}_4|\\
& = (n-2)!-(n-4)(d_{n-3}+2d_{n-4}+d_{n-5})\\
& = (1-1/e+o(1))(n-2)!.\end{align*}
We now prove the following.

\begin{claim}
\label{claim:fix}
Every permutation in \(\mathcal{A} \setminus \mathcal{A}_{1 \mapsto 1,2 \mapsto 2}\) must fix \(\{5,6,\ldots,n\}\) pointwise.
\end{claim}
\begin{proof}[Proof of claim.]
Suppose for a contradiction that some permutation \(\rho \in \mathcal{A}\setminus \mathcal{A}_{1 \mapsto 1,2 \mapsto 2}\) does not fix \(\{5,6,\ldots,n\}\) pointwise; without loss of generality, we may assume that \(\rho(5) \neq 5\). We assert that this forces \(\mathcal{A} \setminus \mathcal{A}_{1 \mapsto 1,2 \mapsto 2}\) to contain a permutation \(\pi\) with at most \(\lfloor n / 2\rfloor\) fixed points.

Indeed, if \(\rho\) itself has at most \(\lfloor n / 2\rfloor\) fixed points, we are done already, so we may assume that $\rho$ has more than \(\lfloor n / 2\rfloor\) fixed points. Let
\[F(\rho) = \{i \geq 5:\ \rho(i)=i\}.\]
We have \(|F(\rho)| > \lfloor n/2\rfloor -4\).

Consider the family
\[\mathcal{M} = \{\sigma \in S_{n}:\ \sigma \textrm{ fixes 1, 2, 5 and some } i \in F(\rho),\ \sigma(i) \neq \rho(i)\ \textrm{for all other }i \geq 3\}.\]
Observe that \(\mathcal{M} \subset \mathcal{G}_{4}\), since each permutation in \(\mathcal{G}\) has at least two fixed points \(\geq 5\). However, each \(\sigma \in \mathcal{M}\) has \(|\sigma \cap \rho|=1\), since by (\ref{eq:12condition}), \(\rho\) cannot fix 1 or 2. Hence,
\[\mathcal{A}_{1 \mapsto 1,2 \mapsto 2} \subset \mathcal{G}_{4} \setminus \mathcal{M}.\]
Observe that
\[|\mathcal{M}| \geq |F(\rho)| d_{n-4} \geq \tfrac{1}{2}(1/e-o(1))(n-3)!,\]
so by (\ref{eq:lowerboundA}), we see that
\[|\mathcal{A} \setminus \mathcal{A}_{1 \mapsto 1,2 \mapsto 2}| \geq \tfrac{1}{2}(1/e-o(1))(n-3)!.\]
As observed above, there are at most \(n!/(\lfloor n / 2 \rfloor)!\) permutations with at least \(\lfloor n/2 \rfloor\) fixed points, so there must be some permutation \(\pi \in \mathcal{A} \setminus \mathcal{A}_{1 \mapsto 1,2 \mapsto 2}\) with at most \(\lfloor n / 2\rfloor\) fixed points, as asserted.

By exactly the same argument as in the proof of Claim \ref{claim:H}, this implies that
\begin{align*} |\mathcal{A}_{1\mapsto 1,2\mapsto 2}| & \leq |\mathcal{G}_4| - \tfrac{1}{8} (1/e +o(1))(n-2)!\\
& = (1-1/e+o(1))(n-2)! - \tfrac{1}{8}(1/e+o(1))(n-2)!\\
& = (1-1/e-1/(8e)+o(1))(n-2)!,\end{align*}
and therefore
\begin{align*} |\mathcal{A}| & \leq (1-1/e-1/(8e)+o(1))(n-2)!+|\mathcal{A}\setminus\mathcal{A}_{1 \mapsto 1,2 \mapsto 2}|\\
& \leq (1-1/e-1/(8e)+o(1))(n-2)! + O((n-3)!)\\
& \leq (1-1/e-1/(8e)+o(1))(n-2)!,\end{align*}
contradicting (\ref{eq:lowerboundA}). The claim is proved.
\end{proof}

Claim \ref{claim:fix} says that
\[\mathcal{A} \setminus \mathcal{A}_{1 \mapsto 1, 2\mapsto 2} \subset S_{[4]}.\]
By (\ref{eq:12condition}), we have
\[\mathcal{A} \setminus \mathcal{A}_{1 \mapsto 1,2 \mapsto 2} \subset \{(1 \ 3)(2\ 4),(1\ 4)(2\ 3),(1\ 3\ 2\ 4),(1\ 4\ 3\ 2)\}.\]
Combining this with (\ref{eq:containmentfor1122}), we see that
\[\mathcal{A} \subset \mathcal{G}_{4} \cup \{(1 \ 3)(2\ 4),(1\ 4)(2\ 3),(1\ 3\ 2\ 4),(1\ 4\ 3\ 2)\} = \mathcal{B}.\]
Hence, by the maximality of \(\mathcal{A}\), \(\mathcal{A} = \mathcal{B}\), proving the theorem.
\end{proof}

\end{document}